\documentclass{amsproc}

\usepackage{tikz}
\usetikzlibrary{automata,cd,shapes}
\usepackage{faktor} 
\usepackage{algorithm}
\usepackage{algpseudocode}
\usepackage{amssymb}
\usepackage{mathtools}
\usepackage[T1]{fontenc}

\usepackage{}

\newtheorem{theorem}{Theorem}

\newtheorem{lemma}{Lemma}[section]
\newtheorem{proposition}[lemma]{Proposition}

\theoremstyle{definition}

\theoremstyle{remark}
\newtheorem{remark}[lemma]{Remark}

\numberwithin{equation}{section}

\DeclareMathOperator{\Sym}{Sym}

\DeclareMathOperator{\PSL}{PSL}
\DeclareMathOperator{\CR}{CR}
\DeclareMathOperator{\tr}{tr}
\newcommand\uEA{\mathcal{EA}}
\newcommand\C{\mathbb{C}}

\newcommand\N{\mathbb{N}}

\newcommand\D{\mathbb{D}}

\newcommand\PP{\mathbf{P}} 
\newcommand\BP{\mathbf{P}}

\newcommand{\TotalDiff}[2]{\mathbf{D}(#1)_{|#2}}
\newcommand{\EA}{\operatorname{EA}}
\newcommand{\W}{\operatorname{W}}
\newcommand{\JEA}{\operatorname{JEA}}
\newcommand{\GSW}{\operatorname{GSW}}
\begin{document}

\title{Diverging orbits for the Ehrlich--Aberth and the Weierstrass root finders}

\author{Bernhard Reinke}
\address{Institut de Mathématiques (UMR CNRS7373) \\
Campus de Luminy \\
163 avenue de Luminy --- Case 907 \\
13288 Marseille 9 \\
France}
\curraddr{}
\email{}
\thanks{}

\date{\today}

\begin{abstract}
We show that the higher dimensional ``Weierstrass'' and ``Ehrlich--Aberth'' methods for finding roots of polynomials have infinite orbits that diverge to infinity. This is possible for the Jacobi update scheme (all coordinates are updated in parallel) as well as Gauss--Seidel (any coordinate update is used for all subsequent coordinates).
\end{abstract}

\maketitle

\section{Introduction}
Finding roots of univariate complex polynomials numerically is one of the fundamental problems for numerical
algebraic geometry. %
Many numerical algorithms for finding roots can be interpreted as complex dynamical systems. 
For example, the Newton method
can be understood as a one-dimensional holomorphic dynamical system, and this was used in \cite{NewtonHSS} to give a good set of starting points that are guaranteed to find all roots. Other methods, such as the Ehrlich--Aberth method and the Weierstrass method, give rise to rational maps $\C^n \dasharrow \C^n$ for polynomials of degree $n$. In contrast to the Newton method, these higher-dimensional methods approximate all roots at the same time.

So far, not much is known about the global dynamics of higher-dimensional root-finding methods. In \cite{reinke_weierstrass_2020}, two phenomena were established for the Jacobi variant of the Weierstrass method: the existence of attracting cycles (and thereby non-general convergence), and the
existence of diverging orbits to infinity. This is a stark contrast to the Newton method: there $\infty \in \BP^1$ is a repelling fixed point.

In this paper we continue the investigation of diverging orbits. We show the following:
\begin{theorem}
  For every degree $d\geq 4$, there is a polynomial $\tilde{p}$ of degree $d$ with distinct roots such that the Jacobi variant of the Ehrlich--Aberth method for $\tilde{p}$ has diverging orbits in $\C^d$.
  \label{cor:eadeg}
\end{theorem}

\begin{theorem}
For every polynomial $\tilde{p}$ with distinct roots of degree $d \geq 3$, the Gauss--Seidel variant of the Weierstrass method for $\tilde{p}$ has diverging orbits in $\C^d$.
  \label{cor:wsdeg}
\end{theorem}

In Section~\ref{sec:hd} we provide common background of the discussed root-finding methods and a discussion for the various variants on how to use them.
We discuss the Jacobi variant of the Ehrlich--Aberth method in Section~\ref{sec:ea} and the Gauss--Seidel variant of the Weierstrass method in Section~\ref{sec:ws}.

\emph{Acknowledgements.} We gratefully acknowledge support by the Advanced Grant HOLOGRAM by the European Research Council. 
The proof of both theorems were obtained with the help of computer algebra systems, in particular the \emph{Maxima}~\cite{Maxima} kernel of \emph{SageMath}~\cite{sagemath} and \emph{Singular}~\cite{DGPS}.

\section{Higher-dimensional root-finding methods}
\label{sec:hd}
In this section we define the maps under consideration. Let $p \in \C[z]$ be a monic polynomial of degree $n$. The root-finding methods
considered in this paper try to find all $n$ roots at the same time, by doing an iteration on $\C^n$. In fact, we
will consider the root-finding methods as rational maps $F \colon \C^n \dashrightarrow \C^n$, and try to extend them to sensible completions of $\C^n$.

We have an update function $g_p \colon \C \times \C^{n-1} \rightarrow \C$ where heuristically
$g_p(z_i;z_1,\dots, \hat{z_i},\dots z_n)$ tries to find a better approximation for $z_i$ to a root of $p$ under the assumption that all other coordinates
$z_1,\dots, \hat{z_i},\dots z_n$ are already close to roots of $p$. We will only consider update functions which are symmetric in the $\C^{n-1}$ component.

We consider two kinds of using the update function $g_p$ on an approximation vector $(z_1,\dots,z_n)$. One possibility is to apply $g_p$ simultaneously on every component, this is called  the Jacobi variant. Since we assume that $g_p$ is symmetric in the $\C^{n-1}$ component, the Jacobi variant is $S_n$ equivariant.

Another possibility is to apply the update function one at a time, and using the results already for the coming coordinates. This is called the Gauss--Seidel variant. See the algorithm listings. Our naming scheme comes from classical methods in numerical linear algebra.%

\begin{algorithm}
  \caption{Jacobi variant}
  \begin{algorithmic}
   \Function{Jacobi}{$z_1,\dots,z_n$}
   \For{$i \in 1,\dots,n$}
   \State $z'_i \leftarrow g_p(z_i;z_1,\dots, \hat{z_i},\dots z_n)$
   \EndFor 
    \State return $(z'_1,\dots,z'_n)$ 
   \EndFunction
  \end{algorithmic}
  \label{algo:j}
\end{algorithm}

\begin{algorithm}
  \caption{Gauss--Seidel variant}
  \begin{algorithmic}
    \Function{GaussSeidel}{$z_1,\dots,z_n$}
    \For{$i \in 1,\dots,n$}
    \State $z_i \leftarrow g_p(z_i;z_1,\dots, \hat{z_i},\dots z_n)$
   \EndFor 
    \State return $(z_1,\dots,z_n)$ 
   \EndFunction
  \end{algorithmic}
  \label{algo:gs}
\end{algorithm}

Both the Weierstrass method and Ehrlich--Aberth method use the rational function

\[q_i(z) = \frac{p(z)}{\prod_{j\not=i}(z-z_j)}\] as basis for their heuristic step. If $p(z) = \prod_j (z - \alpha_j)$ and all $z_j$ apart from $z_i$ are ``close'' to roots $\alpha_j$ of $p$, then
  $q_i$ is ``close'' to the linear function $(z - \alpha_i)$. The Weierstrass update $\W_p(z_i;z_1,\dots, \hat{z_i},\dots z_n)$ tries to solve $q_i$, pretending that it is a linear function, so 
  \begin{equation}
    \W_p(z_i;z_1,\dots, \hat{z_i},\dots z_n) = z_i - q_i(z_i)\text{.}
    \label{def:wsstep}
  \end{equation}

  For the Ehrlich--Aberth function, we do a Newton step
  
  \begin{equation}
    \EA_p(z_i;z_1,\dots, \hat{z_i},\dots z_n) = z_i - \frac{q_i(z_i)}{{q'}_i(z_i)}\text{.}
    \label{def:eastep}
  \end{equation}
\begin{remark}
  We can also think of the Gauss--Seidel variant as the $n$-th iterate of the function
  $R_p \colon (z_1,\dots,z_n) \mapsto (z_2,\dots,z_n,g_p(z_1;z_2,\dots,z_n))$. This is a kind of cyclic shift,
  it takes the first coordinate, applies the step function, and moves it to the end. Since we assume that the step functions
  are symmetric in the additional coordinates, it is easy to see that $R^n_p$ is indeed the whole Gauss-Seidel step.
  \label{rem:cyclic}
  \end{remark}
\section{Ehrlich--Aberth}
\label{sec:ea}
\subsection{Möbius invariance}

\begin{lemma}[Field interpretation]
  Let $p(z) = \prod_i (z - \alpha_i)$ and $(z_1,\dots,z_n) \in \C^n$. 
  Generically, 
  \begin{equation}
    z_i' = \EA_p(z_i;z_1,\dots, \hat{z_i},\dots z_n) \Leftrightarrow \frac{1}{z_i - z_i'} = \sum_j \frac{1}{z_i - \alpha_j} - \sum_{j\not=i} \frac{1}{z_i - z_j}
    \label{eqn:field}
  \end{equation}
  \label{lem:field}
\end{lemma}
\begin{proof}
  Note that $q_i / q'_i$ is the inverse of the logarithmic derivative of $q_i$, so 
  \[
    q_i(z_i) / q'_i(z_i) = \frac{1}{\sum_j \frac{1}{z_i - \alpha_j} - \sum_{j\not=i} \frac{1}{z_i - z_j}}\text{.}
  \]
  The claim now follows from elementary algebra.
\end{proof}
Without the second sum in the right hand side of \eqref{eqn:field}, we would obtain the update step for the classical Newton method. We can think of the Ehrlich--Aberth method as a correction of the Newton method, where we make add a term to avoid that coordinates converge to the same root. See Aberth's derivation~\cite{Aberth1973} for an electrostatic interpretation of Lemma~\ref{lem:field}.
\begin{lemma}[Möbius invariance]
  Let $M(z) = \frac{az+b}{cz+d}$ be a Möbius transformation, let $p(z) = \prod_i (z - \alpha_i)$,
  let $\tilde{p}(z) = \prod_i (z - M(\alpha_i))$,
  then $M(\EA_p(z_i;z_1,\dots, \hat{z_i},\dots,z_n)) = \EA_{\tilde{p}}(M(z_i);M(z_1),\dots, \hat{M(z_i)},\dots M(z_n))$
as rational functions in $z_i;z_1,\dots, \hat{z_i},\dots, z_n$.
  \label{lem:moebinv}
\end{lemma}
\begin{proof}
  By the previous lemma, it is enough to show that
  \[ 
  \frac{1}{z_i - z_i'} = \sum_j \frac{1}{z_i - \alpha_j} - \sum_{j\not=i} \frac{1}{z_i - z_j} \Leftrightarrow
  \frac{1}{M(z_i) - M(z_i')} = \sum_j \frac{1}{M(z_i) - M(\alpha_j)} - \sum_{j\not=i} \frac{1}{M(z_i) - M(z_j)}
\]
holds generically. This is a straightforward computation for a generating set of Möbius transformations such as translations, scalings and the involution $z \mapsto z^{-1}$.
\end{proof}
Motivated by the previous two lemmas, it makes sense to extend $\EA_p$ to a rational function $\EA_p \colon \PP^1 \times \left( \PP^1 \right)^{n-1} \rightarrow \PP^1$. One way to make this precise is to consider the ``universal Ehrlich--Aberth'' map: %
\begin{lemma}
  \label{lem:indet}
  The universal Ehrlich-Aberth map
  \begin{align*}
    \uEA \colon \PP^1 \times \left( \PP^1 \right)^{n-1} \times  \left( \PP^1 \right)^{n} &\dasharrow \PP^1 \\
    (z_1;z_2,\dots,z_n;\alpha_1,\dots,\alpha_n) &\mapsto EA_p (z_1;z_2,\dots,z_n) \text{ for } p(z) = \prod_i (z-\alpha_i)
  \end{align*}
  is Möbius equivariant. We have
  \begin{equation}
    \label{eqn:eainf}
    \uEA(\infty;w_2,\dots,w_n;\alpha_1,\dots,\alpha_n) = \sum^n_{i=1} \alpha_i - \sum^{n-1}_{j=1} w_j
  \end{equation}
  In particular, $\uEA(z_1;z_2,\dots,z_n;\alpha_1,\dots,\alpha_n) = z_1$ if $z_1$ is equal to exactly one of
  $z_2,\dots,z_n;\alpha_1,\dots,\alpha_n$, and the indeterminacy locus of $\uEA$ is contained in the
  union of intersections of two diagonals of the form $z_1 = z_j$ or $z_1 = \alpha_i$.
\end{lemma}
\begin{proof}
  Möbius equivariance is a reformulation of Lemma~\ref{lem:moebinv}.
  Equation~\eqref{eqn:eainf} then follows from equation~\eqref{eqn:field} by applying the Möbius transformation $z \mapsto z^{-1}$ for $z_1 = 0$.

  The right hand side of equation~\eqref{eqn:eainf} is well defined if at most one of the summands is equal to $\infty$.
  From this, and Möbius equivariance, the rest follows.
\end{proof}

One should note that we can now move some roots of the polynomials to $\infty$. In this case $\tilde{p}$ has lower degree then dimension $n$ of the dynamical space,
the formula \ref{def:eastep} still make sense in this case. 

\subsection{Periodic Points of Permutation Type}
For a polynomial $p$ of degree $\leq n$, we denote the Jacobi variant using the Ehrlich--Aberth step associated to $p$ by $\JEA_p$.
So $\JEA_p \colon \left( \PP^1 \right)^{n} \dashrightarrow \left( \PP^1 \right)^{n}$. 
By applying Lemma~\ref{lem:indet} componentwise, we see that the indeterminacy locus is contained in
the union of intersections of two diagonals of the form $z_i = z_j$ or $z_i = \alpha_j$.

  The easiest orbits to understand are the ones where one Jacobi step just
  permutes the coordinates. We will give an example with a nice geometric
  interpretation in the next section.
  However, we cannot hope to use them to find attracting cycles. Here is why:
  \begin{lemma}
    Let $p$ be a monic polynomial of degree $n$, let $z_1, \dots, z_n \in \PP^1$ be pairwise different.
    Assume none of the $z_i$ are roots of $p$ and that $\JEA_p(z_1, \dots , z_n) = \sigma_* (z_1, \dots , z_n)$ for some $\sigma \in \Sym(n)$.
    Then $\tr \TotalDiff{\sigma^{-1}_* \circ \JEA_p}{(z_1, \dots, z_n)} = -n$.
    In particular, the periodic orbit of $(z_{1}, \dots, z_n)$ is not attracting.
        \label{lem:trace}
  \end{lemma}
  \begin{proof}
    Since we assume that $z_1 \dots z_n$ are pairwise different and none of the $z_i$ are roots of $p$, the permutation $\sigma$ does not have a fixed point.
    By Möbius invariance, we can also assume that all $z_i$ are in $\C$. We claim that the matrix $\TotalDiff{\sigma^{-1}_* \circ \JEA_p}{(z_1, \dots, z_n)}$
    in the basis $\frac{\partial}{\partial z_1},\dots,\frac{\partial}{\partial z_n}$ has all diagonal entries equal to $-1$.
    For this it is enough to show the following: if $\EA_p(z_1;z_2,\dots,z_n) = z_2$, then $\frac{\partial \EA_p(z_1;z_2,\dots,z_n)}{\partial z_2} = -1$. 
    This follows directly from Lemma~\ref{lem:field} by differentiating the field interpretation~\eqref{eqn:field} by $z_2$ and $z_1'$.

    In particular, the eigenvalues of $\TotalDiff{\sigma^{-1}_* \circ \JEA_p}{(z_1, \dots, z_n)}$
    can not all lie in $\D$. As the differential of the first return map for $(z_1, \dots, z_n)$ is a power of $\TotalDiff{\sigma^{-1}_* \circ \JEA_p}{(z_1, \dots, z_n)}$, the periodic orbit of $(z_{1}, \dots, z_n)$ is not attracting.
  \end{proof}
\subsection{Two-cycles via harmonic quadruples}
The idea in this section is to construct a number of periodic points of period 2 for generic degree 4 polynomials such that they have a nice description in terms of Möbius transformations.

We will use the following definition of the cross-ratio
\[\CR(a,b,c,d) = \frac{(a-c)(b-d)}{(a-d)(b-c)}\text{.}\]
The cross-ratio is well defined whenever $\left\{a,b,c,d\right\} \subset \PP^1$ has cardinality at least 3.
It is well known that the cross-ratio is invariant under Möbius transformations.
In our normalization $\CR(0,\infty,c,d) = c/d$.

For $a \not= b \in \PP^1$ let $M_{a,b} \in \PSL(2,\C)$ the unique Möbius involution which fixes $a$ and $b$.
It is easy to see that this is well-defined, for example $M_{0,\infty}(z) = -z$, so if $N \in \PSL(2,\C)$ moves
$a$ to $0$ and $b$ to $\infty$, then $M_{a,b} = N^{-1} \circ M_{0,\infty} \circ N$.
\begin{lemma}
  Let $a,  b, c, d \in \PP^1$ be pairwise different. 
  The following are equivalent:
  \begin{enumerate}
    \item The cross-ratio $\CR(a,b,c,d)$ is equal to $-1$.
      \label{cr}
    \item $M_{a,b}$ interchanges $c$ and $d$.
      \label{interchange}
    \item $M_{c,d}$ interchanges $a$ and $b$.
  \end{enumerate}
We say in this case
that $a,b,c,d$ are harmonic.
\label{lem:harmonic}
\end{lemma}
\begin{proof}
  Since $\CR(a,b,c,d) = \CR(c,d,a,b)$, it is enough to show that (\ref{cr}) and (\ref{interchange}) are equivalent.
  By Möbius invariance, we can assume wlog.\ that $a = 0$ and $b = \infty$. In this case $M_{0,\infty}(z) = -z$,
  and $\CR(0,\infty,c,d) = c/d$, so the statement is obvious.
\end{proof}
\begin{remark}
  If $a,b,c,d$ are harmonic, they must lie on a real circle or a real line. If $a,b,c,d$ lie on a
real line, this definition extends the classical notion of harmonic quadruples. It is
easy to check that this is really a property of the unordered pair of unordered
pairs $\left\{ \left\{ a,b \right\},\left\{ c,d \right\} \right\}$.

Since cross-ratios are invariant under Möbius transformations, so is the notion of harmonic quadruples.
  \label{rem:harmonic}
\end{remark}
\begin{lemma}
  Let $a,  b, c, d \in \PP^1$ be pairwise different. Then there is a unique unordered pair $\left\{ e,f \right\}$ such that $a,b$ is harmonic to $e,f$ and $c,d$ is harmonic to $e,f$.  \label{lem:harmonicint}
\end{lemma}
\begin{proof}
  Using criterion (\ref{interchange}) of Lemma~\ref{lem:harmonic}, it is straightforward to verify that $e,f$ are the two fixed points of $M_{a,b} \circ M_{c,d}$.
\end{proof}
\begin{proposition}
  Let $p(z) = \prod^4_{i=1}(z-\alpha_i)$ be a polynomial with distinct roots. Let $(z_1,z_2,z_3,z_4)$ be such that
  \begin{itemize}
    \item
      $\left\{ z_1,z_2 \right\}$ is the unique
      unordered pair harmonic to $\left\{ \alpha_1, \alpha_2 \right\}$ and $\left\{ \alpha_3, \alpha_4 \right\}$,
     \item $\left\{ z_3, z_4 \right\}$
      is the unique unordered pair harmonic to $\left\{ \alpha_1, \alpha_3 \right\}$ and $\left\{ \alpha_2, \alpha_4 \right\}$. 
  \end{itemize}
  Then under the Jacobi variant of the Ehrlich--Aberth method, $(z_1, z_2, z_3, z_4)$ is sent to $(z_2, z_1, z_4, z_3)$.
  \label{thm:geometricintr}
\end{proposition}
\begin{proof}
  By symmetry, it is enough to show that $\EA_p(z_1;z_2,z_3,z_4) = z_2$. By Möbius invariance, we can assume that $z_1 = \infty$, $z_2 = 0$, and
  in this case $\EA_p(\infty; 0;z_3,z_4) = (\sum^4_{i=1} \alpha_i) - z_3 - z_4$. Now $\left\{ \infty, 0 \right\}$ is harmonic to $\left\{ \alpha_1, \alpha_2 \right\}$ and harmonic to $\left\{ \alpha_3, \alpha_4 \right\}$, so $M_{0,\infty}(z) = -z$ both interchanges $\alpha_1$ and $\alpha_2$ as well as $\alpha_3$ and $\alpha_4$. Hence $M_{0,\infty}(z)$ fixes the unordered pair of unordered pairs $\left\{ \left\{ \alpha_1, \alpha_3 \right\}, \left\{ \alpha_2, \alpha_2 \right\}  \right\}$ set-wise, so  $M_{0,\infty}$ must also interchange $z_3$ and $z_4$. So $\alpha_1 = -\alpha_2, \alpha_3 = - \alpha_4, z_3 = -z_4$ and thus
  $\EA_p(\infty; 0,z_3,z_4) = 0$.
\end{proof}
\begin{remark}
  Using this proposition, we can produce $\binom{4}{2}\cdot2\cdot2 = 24$ periodic points for the Jacobi variant of the Ehrlich--Aberth method with mapping behavior
  $(z_1, z_2, z_3, z_4) \mapsto (z_2, z_1, z_4, z_3)$. These are far from the only ones, computer algebra shows that there are generically 72 such periodic points. We are not aware of a description of the 48 remaining ones that is as nice as given in this theorem.
\end{remark}
\subsection{Diverging Orbits for Jacobi Ehrlich--Aberth}
By Proposition~\ref{thm:geometricintr}, we have a nice class of periodic orbits for the Jacobi Ehrlich--Aberth method. Here is a particularly nice parametrization:
\begin{lemma}
  Let $\lambda \in \C \setminus \left\{ -1, 0, 1 \right\}$. For $p_\lambda(z) = (z^2-1)(z^2-\lambda^4)$, the points $(\lambda,-\lambda,\sqrt{-1}\lambda,-\sqrt{-1}\lambda)$
  and $(\lambda,-\lambda,0,\infty)$ are periodic points of period 2 with mapping scheme $(z_1, z_2, z_3, z_4) \mapsto (z_2, z_1, z_4, z_3)$.
  \label{lem:param}
\end{lemma}
\begin{proof}
  Straight forward computations using Lemma~\ref{lem:harmonic} show that
  $\left\{ \lambda, -\lambda \right\}$ is harmonic to $\left\{ 1,\lambda^2
  \right\}$ and harmonic to $\left\{ -1,-\lambda^2 \right\}$; that $\left\{
  \sqrt{-1}\lambda, -\sqrt{-1}\lambda \right\}$ is harmonic to $\left\{ -1,
  \lambda^2 \right\}$ and harmonic to $\left\{ 1, -\lambda^2 \right\}$; and
  that $\left\{ 0,\infty \right\}$ is harmonic to $\left\{ 1,-1 \right\}$ and harmonic to $\left\{ \lambda^2,-\lambda^2 \right\}$.
  The claim now follows from Proposition~\ref{thm:geometricintr}.
\end{proof}
We can explicitly compute the first return maps for these periodic points. 
\begin{proposition}[Two-cycles via harmonic configurations]
  Let $p$ be a polynomial of degree $4$ such that its roots form a parallelogram $\alpha_1,\dots,\alpha_4$.
  Then the Jacobi Ehrlich--Aberth method of degree $4$ for $p$ has a two-cycle in $(\PP^1)^4 \setminus \C^4$. Moreover, there is 
  a nonempty open subset $U \subset \C$ such that if $\CR(\alpha_1,\alpha_2,\alpha_3,\alpha_4) \in U$, then this cycle has a stable manifold intersecting $\C^4$, hence the Jacobi Ehrlich--Aberth method has orbits diverging off to infinity.
  \label{thm:escorbits}
\end{proposition}
\begin{proof}
  The statement of the proposition is invariant under affine Möbius
  transformations. We can move any parallelogram to
  $(1,\lambda^2,-1,-\lambda^2)$ by an affine Möbius transformation, so it is
  enough to consider $p_\lambda(z) = (z^2-1)(z^2-\lambda^4)$ as in
  Lemma~\ref{lem:param}. In this case, we already know that
  $(\lambda,-\lambda,0,\infty)$ is part of a two-cycle in $(\PP^1)^4 \setminus
  \C^4$, so it remains to show for some open set $V \subset \C$, this cycle has
  a stable manifold intersecting $\C^4$ for $\lambda \in V$. The set $U$ will
  be the image of $V$ under $\lambda \mapsto CR(1,\lambda^2,-1,-\lambda^2)$.

  Let $K \colon (\PP^1)^4 \mapsto (\PP^1)^4$ be the permutation $(z_1, z_2, z_3, z_4) \mapsto (z_2, z_1, z_4, z_3)$. Then $(\lambda,-\lambda,0,\infty)$ is a fixed point of $K \circ \EA_{p_{\lambda}}$. We take as basis for the tangent space at the point $\partial z_1, \partial z_2, \partial z_3, z^{-2}_4 \partial z_4$, where $\partial z_i$ is the extension from $\C^4$ to $(\PP^1)^4$ of the constant vector field in direction of the $i$-th coordinate.
  In this basis, the differential $D(K \circ \EA_{p_{\lambda}})$ at $(\lambda,-\lambda,0,\infty)$ is given by
  \begin{equation}
    \left(\begin{array}{rrrr}
        -1 & -\frac{2 \, {\left({\lambda}^{4} + 14 \, {\lambda}^{2} + 1\right)}}{{\left({\lambda} + 1\right)}^{2} {\left({\lambda} - 1\right)}^{2}} & -4 & -4 \, {\lambda}^{2} \\
        -\frac{2 \, {\left({\lambda}^{4} + 14 \, {\lambda}^{2} + 1\right)}}{{\left({\lambda} + 1\right)}^{2} {\left({\lambda} - 1\right)}^{2}} & -1 & -4 & -4 \, {\lambda}^{2} \\
        -1 & -1 & -1 & -2 \, {\lambda}^{4} + 2 \, {\lambda}^{2} - 2 \\
        -\frac{1}{{\lambda}^{2}} & -\frac{1}{{\lambda}^{2}} & -\frac{2 \, {\left({\lambda}^{4} - {\lambda}^{2} + 1\right)}}{{\lambda}^{4}} & -1
    \end{array}\right)\text{.}
    \label{eqn:jacobian}
  \end{equation}

We use Sage to compute an explicit eigenspace decomposition:
we give a eigenspace decomposition of the matrix~\eqref{eqn:jacobian}, in dependence of $\lambda$.
Let $D = {{\lambda}^{16} - 8 \, {\lambda}^{14} + 12 \, {\lambda}^{12}
                + 8 \, {\lambda}^{10} + 230 \, {\lambda}^{8} + 8 \, {\lambda}^{6} + 12 \, {\lambda}^{4}
              - 8 \, {\lambda}^{2} + 1}$.
  This matrix has eigenvalues
  \begin{equation}
    \frac{{\lambda}^{4} + 30 \, {\lambda}^{2} + 1}{{\lambda}^{4} - 2 \, {\lambda}^{2} + 1}, -\frac{{\lambda}^{8} - {\lambda}^{6} + 16 \, {\lambda}^{4} - {\lambda}^{2} + \sqrt{D} + 1}{{\lambda}^{6} - 2 \, {\lambda}^{4} + {\lambda}^{2}}, -\frac{{\lambda}^{8} - {\lambda}^{6} + 16 \, {\lambda}^{4} - {\lambda}^{2} - \sqrt{D} + 1}{{\lambda}^{6} - 2 \, {\lambda}^{4} + {\lambda}^{2}}, \frac{2 \, {\lambda}^{4} - 3 \, {\lambda}^{2} + 2}{{\lambda}^{2}}
  \end{equation}
  with right eigenvectors given by the columns of the following matrix:
  \begin{equation}
\left(\begin{array}{rrrr}
1 & 1 & 1 & 0 \\
-1 & 1 & 1 & 0 \\
0 & \frac{{\lambda}^{8} - 4 \, {\lambda}^{6} - 10 \, {\lambda}^{4} - 4 \, {\lambda}^{2} + \sqrt{D} + 1}{8 \, {\left({\lambda}^{6} - 2 \, {\lambda}^{4} + {\lambda}^{2}\right)}} & \frac{{\lambda}^{8} - 4 \, {\lambda}^{6} - 10 \, {\lambda}^{4} - 4 \, {\lambda}^{2} - \sqrt{D} + 1}{8 \, {\left({\lambda}^{6} - 2 \, {\lambda}^{4} + {\lambda}^{2}\right)}} & 1 \\
0 & \frac{{\lambda}^{8} - 4 \, {\lambda}^{6} - 10 \, {\lambda}^{4} - 4 \, {\lambda}^{2} + \sqrt{D} + 1}{8 \, {\left({\lambda}^{8} - 2 \, {\lambda}^{6} + {\lambda}^{4}\right)}} & \frac{{\lambda}^{8} - 4 \, {\lambda}^{6} - 10 \, {\lambda}^{4} - 4 \, {\lambda}^{2} - \sqrt{D} + 1}{8 \, {\left({\lambda}^{8} - 2 \, {\lambda}^{6} + {\lambda}^{4}\right)}} & -\frac{1}{{\lambda}^{2}}
\end{array}\right)
  \end{equation}
  For $\lambda = \pm\frac{1}{2} \, \sqrt{\pm\sqrt{-7} + 3}$, we find that the eigenvalues
  are $-63, 61.2529526\dots, -2.2529526\dots$ and $0$. By small (but nonzero) perturbation of $\lambda$, we thereby have
  an open set $V \subset \C$ where the first three eigenvalues lie in $\C \setminus \overline \D$, and the last eigenvalue is in $\D^*$.
  In this case, the point $(\lambda,-\lambda,0,\infty)$ is hyperbolic for the  map $K \circ \JEA$ with one attracting direction,
  so by the stable manifold theorem (see for example \cite[Ch.~2, Section~6]{PalisdeMelo}), we have a complex one-dimensional manifold that
  is tangential to $(0,0,1,-\frac{1}{{\lambda}^{2}})$, so in a neighborhood, points on the stable manifold give rise to diverging orbits in $\C^4$.
\end{proof}
\subsection{Extension to higher degrees}
We want to extend this result to higher degrees. We can use the following embedding:
\begin{lemma}
 Let $p$ be a polynomial of degree $d$, let $\alpha \in \C$ and let $\tilde{p}(z) =(z - \alpha)p(z)$.
We have commutative diagrams of rational maps
      \begin{tikzcd}[ampersand replacement=\&]
        (\BP^1)^d \ar[d,"\iota_\alpha"] \ar[r,dashed,"\JEA_p"] \& (\BP^1)^d \ar[d,"\iota_\alpha"] \\
        (\BP^1)^{d+1} \ar[r,dashed,"\JEA_{\tilde p}"] \& (\BP^1)^{d+1}
      \end{tikzcd}
      where $ \iota_\alpha\left( z_1,\dots,z_n \right) = (z_1,\dots,z_n,\alpha)$. Moreover, if $z_1,\dots,z_n \in \BP^1 \setminus \left\{ \alpha \right\}$,
      then $\JEA_p(z_1,\dots,z_n)$ is well defined if and only if $\JEA_{\tilde {p}}(z_1,\dots,z_n,\alpha)$ is.
  \label{ea:degemb}
\end{lemma}
\begin{proof}
  This follows directly from Lemma~\ref{lem:field}.
\end{proof}
\begin{proof}[Proof of Theorem \ref{cor:eadeg}]
  Let $k = d - 4$.
  Let $\lambda \in U$ as in Proposition~\ref{thm:escorbits}. Let $\alpha_1, \dots, \alpha_k \in \C \setminus \left\{ 0, \lambda, -\lambda \right\}$ and
  take $\tilde{p} = \prod_i (z-\alpha_i) p(z)$. We use a repeated application of Lemma~\ref{ea:degemb} to embed the dynamics of $\JEA_p$ into
  $\JEA_{\tilde{p}}$ via the map $(z_1,\dots,z_4) \mapsto (z_1,\dots,z_4,\alpha_1,\dots,\alpha_k)$. In a neighborhood of the periodic point $(0,\infty,\lambda,-\lambda)$, we see that $\JEA_{\tilde{p}}$ is well defined on the image of the stable manifold for $(0,\infty,\lambda,-\lambda)$, and this again gives rise to diverging orbits in $\C^d$. %
\end{proof}
\section{Weierstrass}
For a polynomial $p$ of degree $\leq n$, we denote the Gauss--Seidel variant using the Weierstrass step associated to $p$ by $\GSW_p$,
and the cyclic shift from Remark~\ref{rem:cyclic} for the Weierstrass map for cubic polynomials by $R_p$, so
\begin{eqnarray*}
  R_p(z_1,z_2,z_3) = \left(z_2,z_3,z_1 - \frac{p(z_1)}{(z_1-z_2)(z_1-z_3)}\right)\text{.}
\end{eqnarray*}
\label{sec:ws}
\subsection{Affine invariance}
\begin{lemma}[Affine invariance]
  Let $M(z) = {az+b}$ be an affine transformation let $p(z) = \prod_i (z - \alpha_i)$,
  let $\tilde{p}(z) = \prod_i (z - M(\alpha_i))$,
  then $M(\W_p(z_i;z_1,\dots, \hat{z_i},\dots z_n)) = \W_{\tilde{p}}(M(z_i);M(z_1),\dots, \hat{M(z_i)},\dots M(z_n))$
  as rational functions in $(z_i;z_1,\dots, \hat{z_i},\dots z_n)$
  \label{lem:affinv}
\end{lemma}
\begin{proof}
  Straight forward computation.
\end{proof}
\subsection{Special case $z^3$}
For $p=z^3$, the Gauss--Seidel variant of the Weierstrass map is homogeneous, in the sense that $\GSW_{z^3}(\lambda z_1, \lambda z_2, \lambda z_3) = \lambda \GSW_{z^3}(z_1, z_2, z_3)$. This is also true for the rotational map $F_{z^3}$ from Remark~\ref{rem:cyclic}. So the map $F_{z^3}$ fibers as rational map over $\phi \colon \BP^2 \dasharrow \BP^2$, where
\begin{eqnarray*}
  \phi( [z_1 : z_2 : z_3 ]) = [z_2 : z_3 : z_1 - \frac{z_1^3}{(z_1-z_2)(z_1-z_3)}] = \big[&(z_1-z_2)(z_1-z_3)z_2 \\
    :&(z_1-z_2)(z_1-z_3)z_3 \\
  :&(z_1-z_2)(z_1-z_3)z_1 - z_1^3\big] \text{.}
\end{eqnarray*}
Since $F^3_p = \GSW_p$, the Gauss-Seidel variant of the Weierstrass map for $z^3$ fibers over $\phi^3$.
\begin{lemma}
  The map $\phi$ has a periodic point $[z_1 : z_2 : z_3]$ of period $3$ such that no coordinate
  vanishes, and $\GSW_{z^3}(z_1,z_2,z_3) = \mu (z_1, z_2, z_3)$ for some $\mu \in \C$ with $|\mu| > 1$.
  Moreover, we can choose a periodic point such that the cycle is repelling for $\phi \colon \BP^2 \dasharrow \BP^2$.
  \label{lem:z3esc}
\end{lemma}
\begin{proof}
  We establish this using explicit computations for the map $\GSW_{z^3}$. Let us use variables $z_1,\dots,z_6$. We impose two relations: $(z_4,z_5,z_6)$ should be the
  image of $(z_1,z_2,z_3)$ under the map $\GSW_{z^3}$, and also $(z_4, z_5, z_6)$ should be a scalar multiple of $(z_1,z_2,z_3)$.
  To avoid points of indeterminacy, it is more convenient to impose that $(z_1,z_2,z_3)$ is a scalar multiple of $(z_4,z_5,z_6)$.
  For this we introduce another variable $\lambda$, which will be the inverse of $\mu$.
  Away from points of indeterminacy, the relation $\GSW_{z^3}(z_1,z_2,z_3) = (z_4, z_5, z_6)$ is equivalent to the set of equations
  \begin{eqnarray*}
    (z_1 - z_2)(z_1 - z_3)(z_1 - z_4) - z_1^3  = 0\\
      (z_2 - z_3)(z_2 - z_4)(z_2 - z_5) - z_2^3  = 0\\
      (z_3 - z_4)(z_3 - z_5)(z_3 - z_6) - z_3^3  = 0 \text{.}\\
  \end{eqnarray*}
  The relation $\lambda(z_4,z_5,z_6) = (z_1,z_2,z_3)$ is equivalent to
  \begin{eqnarray*}
      -\lambda z_4 + z_1 = 0 \\
      -\lambda z_5 + z_2  = 0\\
      -\lambda z_6 + z_3 = 0 \text{.}
  \end{eqnarray*}
  We know that the space of possible solutions is homogeneous, so we introduce an extra normalization $z_6 = 1$. Now
  we have a polynomial system of equations that is easily solvable with \emph{Singular} or \emph{Maxima}.
  \emph{Singular} shows this has 18 solutions, and $\lambda$ has to be a root of
  $(\lambda^5 - 2\lambda^4 - 3\lambda^3 + 8\lambda^2 - 4\lambda + 1)(\lambda^3 - 5\lambda^2 + 4\lambda - 1)(\lambda^2 - 3\lambda + 1)$.
  For the factor $(\lambda^2 - 3\lambda +1)$, some coordinates of the cycles are $0$. For the other two factors, we can find solutions
  with $\lambda \in \D^*$, such that all $z_1,\dots,z_6$ are non-zero, and compute the differential of $\phi$ so see that
  there are also solutions that are repelling for $\phi$. %
\end{proof}
\subsection{Diverging Orbits for Gauss--Seidel Weierstrass variant}
Let $p$ be a monic cubic polynomial.
We extend $R_p$ to a map on $\PP^3$. Let $p^*(z,w)$ be the homogenization of $p$, i.e. $p^*(z,w) = p(z/w)w^3$.
Then $R_p$ has the following extension:
\begin{eqnarray*}
  R_p([z_1\colon z_2 \colon z_3 \colon w]) = \big[&(z_1-z_2)(z_1-z_3)z_2 \\
    :&(z_1-z_2)(z_1-z_3)z_3 \\
  :&(z_1-z_2)(z_1-z_3)z_1 - p^*(z_1,w) \\
  :& (z_1-z_2)(z_1-z_3)w\big]
  \label{eqn:R_p_homo}
\end{eqnarray*}
As Gauss--Seidel variant of the Weierstrass map is the third iterate of $R_p$, we focus on
$R_p$ for now.
\begin{lemma} %
  The restriction of $R_p$ to the plane $w = 0$ is independent of $p$ and is given by $\phi$. Moreover,
  if $[z_1:z_2:z_3]$ is periodic for $\phi$ of period $n$, with $F^n_{z^3}(z_1,z_2,z_3) = \mu (z_1,z_2,z_3)$,
  then the eigenvector of $D(F^n_p)$ at $[z_1:z_2:z_3:0]$ transverse to the plane $w = 0$ has eigenvalue $\mu^{-1}$.
  \label{lem:triangluar}
\end{lemma}
\begin{proof}
  On the plane $w = 0$, the Jacobian matrix associated to $R_p$ by equation~\eqref{eqn:R_p_homo} for has a $3 + 1$ triangular block decomposition,
  where the first block is the Jacobian matrix for $\phi$, and the last diagonal entry is independent of $p$.
  From this we see that the eigenvalues of $D(R^n_p)$ for a periodic point $[z_1:z_2:z_3:0]$ are independent of $p$. For $p = z^3$, it is clear that
  the transverse eigenvector has eigenvalue $\mu^{-1}$.
\end{proof}
\begin{proposition}
  For cubic polynomials $p$ with distinct roots, the Gauss--Seidel Weierstrass variant has diverging orbits that go to infinity in every component.
  \label{thm:gsws}
\end{proposition}
\begin{proof}
  In the language of Lemma~\ref{lem:triangluar}, the periodic point constructed in Lemma~\ref{lem:z3esc}
  is a periodic point for $R_p$ on the plane $w=0$ that has an attracting direction transverse to the plane, is repelling on $w = 0$,
  and has no other vanishing coordinate.
  Now we can again apply the stable manifold theorem to see that the points in the stable manifold of the periodic point are the diverging orbits we want.
\end{proof}
\subsection{Extension to higher degrees}

\begin{lemma}
  Let $p$ be a polynomial of degree $d$, let $\alpha \in \C$ and let $\tilde{p}(z) =(z - \alpha)p(z)$. For the Gauss--Seidel variant of the Weierstrass method,
we have commutative diagrams of rational maps
      \begin{tikzcd}[ampersand replacement=\&]
        \C^d \ar[d,"\iota_\alpha"] \ar[r,dashed,"\GSW_p"] \& \C^d \ar[d,"\iota_\alpha"] \\
        \C^{d+1} \ar[r,dashed,"\GSW_{\tilde p}"] \& \C^{d+1}
      \end{tikzcd}
      where $ \iota_\alpha\left( z_1,\dots,z_n \right) = (z_1,\dots,z_n,\alpha)$. Moreover, if $z_1,\dots,z_n \in \C \setminus \left\{ \alpha \right\}$,
      and $\GSW_p(z_1,\dots,z_n)$ is well defined and has no coordinate equal to $\alpha$, then $\GSW_{\tilde {p}}(z_1,\dots,z_n,\alpha)$ is well defined and
      equal to $\iota_\alpha\left( z_1,\dots,z_n \right)$.
  \label{lem:degembws}
\end{lemma}
\begin{proof}
  This follows directly from the defining equation~\ref{def:wsstep}, by stepping through the algorithm~\ref{algo:gs}.
  A key point here is that for the algorithm applied for $\GSW_{\tilde{p}}(z_1,\dots,z_n,\alpha)$, the contribution of $\alpha$ cancels under the assumption $z_i \not= \alpha$, the final step is $W_{\tilde{p}}(\alpha; z'_1,\dots,z'_n)$, there the second part of the assumption is used.
\end{proof}
\begin{proof}[Proof of Theorem~\ref{cor:wsdeg}]
  Let $k = d - 3$ and choose a factorization $\tilde{p}(z) = p(z) \prod_i(z - \alpha_i)$, where $\alpha_1,\dots,\alpha_k \in \C$ are pairwise different.
  By Proposition~\ref{thm:gsws}, there are $z^j_i \in C, 1\leq i \leq 3, j \in \N$, such that $(z^j_1, z^j_2, z^j_3)$ is an infinite forward orbit under the Gauss--Seidel variant for $p$, with $z^j_i \rightarrow \infty$ for all $1 \leq i \leq 3$. By taking a tail of the orbit, we can assume that none of the coordinates $z^j_i$ are equal to roots of $\tilde{p}(z)$. Now we can iteratively apply Lemma~\ref{lem:degembws} to see that $\left( z^j_1, z^j_2, z^j_3, \alpha_1,\dots,\alpha_k \right)$ is a well defined forward orbit for the Gauss--Seidel variant of the Weierstrass method for $\tilde{p}$,
  and the first 3 components go off to infinity.
  \end{proof}
\section{Outlook}
In this paper, we extended the results about diverging orbits for the Weierstrass method by changing
the step function to the Ehrlich--Aberth method or by changing the iteration procedure to the Gauss--Seidel variant.
A natural next step is to consider the Gauss--Seidel variant of the Ehrlich--Aberth method. Based on the Möbius invariance of
the Ehrlich--Aberth method, this is equivalent to finding periodic cycles with at least one attracting direction.
However, we are not aware of explicit descriptions of a family of
periodic points as in the Jacobi case.

Also, the existence of attracting cycles is still open for the various methods discussed here apart from the Jacobi variant of the Weierstrass method.
An intermediate question would be the existence of periodic points with multiple attracting eigenvalues.
\bibliographystyle{alpha}
\bibliography{../bibfile.bib}
\end{document}